\documentclass[a4paper,11pt]{article}
\usepackage{amsmath, amsthm, amssymb}
\usepackage{fancyhdr}
\font\Afont=cmr12 scaled 1000
\font\Bfont=cmr12 scaled 1200
\font\Cfont=cmr12 scaled 1440

\def\R{{\mathbb R}}
\def\N{{\mathbb N}}
\def\Z{{\mathbb Z}}
\def\C{{\mathbb C}}

\def\phii{\varphi}
\newcommand{\esssup}{\operatornamewithlimits{ess\,sup}}

\newtheorem{thm}{Theorem}[section]

\newtheorem{lem}[thm]{Lemma}

\theoremstyle{remark}
\newtheorem{rem}[thm]{Remark}

\theoremstyle{definition}
\newtheorem{dfn}[thm]{Definition}

\numberwithin{equation}{section}

\newlength{\fixboxwidth}
\newcommand{\fix}[1]{\marginpar{\fbox{\parbox{\fixboxwidth}{\raggedright\tiny #1}}}}

\begin{document}
\pagestyle{empty}
\
\vskip1cm
\centerline {\Cfont On sharp embeddings of Besov and Triebel-Lizorkin spaces}
\vskip.5cm
\centerline {\Cfont in the subcritical case}
\vskip2cm
\centerline {\Bfont Jan Vyb\'\i ral}
\vskip.3cm
\centerline {\Afont Mathematisches Institut, Universit\"at Jena}
\centerline {\Afont Ernst-Abbe-Platz 2, 07740 Jena, Germany}
\centerline {email:\ {\tt vybiral@mathematik.uni-jena.de}}
\vskip.5cm
\centerline {\bf \today}
\vskip.5cm
\begin{abstract}
We discuss the growth envelopes of Fourier-analytically defined Besov and 
Triebel-Lizorkin spaces $B^s_{p,q}(\R^n)$
and $F^s_{p,q}(\R^n)$ for $s=\sigma_p=n\max(\frac 1p-1,0)$. These results may be also reformulated
as optimal embeddings into the scale of Lorentz spaces $L_{p,q}(\R^n)$.
We close several open problems outlined already by H.~Triebel in \cite{T-SF} 
and explicitly formulated by D.~D.~Haroske in \cite{H}.
\end{abstract}



\newpage
\pagestyle{fancy}
\section{Introduction and main results}
We denote by $B^s_{p,q}(\R^n)$ and $F^s_{p,q}(\R^n)$ the Fourier-analytic Besov and Triebel-Lizorkin spaces
(see Definition \ref{defsp} for details). The embeddings of these function spaces (and other spaces
of smooth functions) play an important role in functional analysis. If $s>\frac {n}{p}$, then these spaces
are continuously embedded into $C(\R^n)$, the space of all complex-valued bounded and uniformly
continuous functions on $\R^n$ normed in the usual way. If $s<\frac np$ then these function spaces contain 
also unbounded functions. This statement holds true also for $s=\frac np$ under some additional restrictions
on the parameters $p$ and $q$. We refer to \cite[Theorem 3.3.1]{SiTr} for a complete overview.

To describe the singularities of these unbounded elements, we use the technique of the non-increasing
rearrangement.

\begin{dfn}\label{defR}
Let $\mu$ be the Lebesgue measure in $\R^n$.
If $h$ is a measurable function on $\R^n$, we define the non-increasing rearrangement of
$h$ through
\begin{equation}\label{eq':2.1}
h^*(t)=\sup \{\lambda>0: \mu\{x\in\R^n: |h(x)|>\lambda\}>t \},\qquad t\in (0,\infty).
\end{equation}
\end{dfn}

To be able to apply this procedure to elements of $A^s_{p,q}(\R^n)$ (with $A$ standing for $B$ or $F$), 
we have to know whether all the distributions of $A^s_{p,q}(\R^n)$ may be interpreted as measurable functions.
This is the case if, and only if, $A^s_{p,q}(\R^n)\hookrightarrow L_1^{\rm loc}(\R^n)$, the space of 
all measurable, locally-integrable functions on $\R^n$. A complete treatment of this question may
be found in \cite[Theorem 3.3.2]{SiTr}:
\begin{equation}\label{B1}
B^s_{p,q}(\R^n)\hookrightarrow L_1^{\rm loc}(\R^n) \Leftrightarrow
\begin{cases}\text{either}&s>\sigma_p:=n\max(\frac 1p-1,0),\\
\text{or}\quad &s=\sigma_p, 1<p\le\infty, 0<q\le\min(p,2),\\
\text{or}\quad &s=\sigma_p, 0<p\le 1, 0<q\le1
\end{cases}
\end{equation}
and
\begin{equation}\label{F1}
F^s_{p,q}(\R^n)\hookrightarrow L_1^{\rm loc}(\R^n) \Leftrightarrow
\begin{cases}\text{either}&s>\sigma_p,\\
\text{or}\quad &s=\sigma_p, 1\le p<\infty, 0<q\le2,\\
\text{or}\quad &s=\sigma_p, 0<p<1, 0<q\le\infty.
\end{cases}
\end{equation}
Let us assume, that a function space $X$ is embedded into $L_1^{\rm loc}(\R^n)$. 
The {\it growth envelope function} of $X$ was defined by D.~D.~Haroske and H.~Triebel (see \cite{H'}, \cite{H},
\cite{T-SF} and references given there) by
\begin{equation*}
{\mathcal E}^X_G(t):=\sup_{||f|X||\le 1}f^*(t),\quad 0<t<1.
\end{equation*}
If ${\mathcal E}^X_G(t)\approx t^{-\alpha}$ for $0<t<1$ and some $\alpha>0$, then we define the \fix{OK?}
{\it growth envelope index} $u_X$ as the infimum of all numbers $v$, $0<v\le\infty$, such that
\begin{equation}\label{eq:index}
\left( \int_{0}^\epsilon \left[\frac{f^*(t)}{{\mathcal E}^X_G(t)}\right]^{v}\frac{dt}{t}\right)^{1/v}
\le c\, ||f|X||
\end{equation}
(with the usual modification for $v=\infty$) holds for some $\epsilon>0, c>0$ and all $f\in X.$

The pair ${\mathfrak E_G}(X)=({\mathcal E_G^X},u_X)$ is called {\it growth envelope} for the function space $X$.

In the case $\sigma_p<s<\frac np$, the growth envelopes of $A^s_{p,q}(\R^n)$ are known,
cf. \cite[Theorem 15.2]{T-SF} and \cite[Theorem 8.1]{H}. If $s=\sigma_p$ and \eqref{B1} or \eqref{F1} is fulfilled in the $B$ or $F$ case,
respectively, then the known information is not complete, 
cf. \cite[Rem. 12.5, 15.1]{T-SF} and \cite[Prop. 8.12, 8.14 and Rem. 8.15]{H}:

\begin{thm} \label{thm1.1}
(i) Let $1< p<\infty$ and $0<q\le \min(p,2)$. Then 
\begin{equation*}
 \mathfrak E_G(B^0_{p,q})=(t^{-\frac{1}{p}},u)\qquad\text{with}\quad q\le u\le p.
\end{equation*}
(ii) Let $1\le p<\infty$ and $0<q\le 2$. Then 
\begin{equation*}
 \mathfrak E_G(F^0_{p,q})=(t^{-\frac{1}{p}},p).
\end{equation*}
(iii) Let $0<p\le 1$ and $0<q\le 1$. Then 
\begin{equation*}
 \mathfrak E_G(B^{\sigma_p}_{p,q})=(t^{-1},u)\qquad\text{with}\quad q\le u\le 1.
\end{equation*}
(iv) Let $0<p<1$ and $0<q\le\infty$. Then
\begin{equation*}
 \mathfrak E_G(F^{\sigma_p}_{p,q})=(t^{-1},u)\qquad\text{with}\quad p\le u\le 1.
\end{equation*}
\end{thm}

We fill all the above mentioned gaps.
\begin{thm} \label{thm1.2}
(i) Let $1\le p<\infty$ and $0<q\le \min(p,2)$. Then 
\begin{equation*}
 \mathfrak E_G(B^0_{p,q})=(t^{-\frac{1}{p}},p).
\end{equation*}
(ii) Let $0<p<1$ and $0<q\le 1$. Then 
\begin{equation*}
 \mathfrak E_G(B^{\sigma_p}_{p,q})=(t^{-1},q).
\end{equation*}
(iii) Let $0<p<1$ and $0<q\le\infty$. Then
\begin{equation*}
 \mathfrak E_G(F^{\sigma_p}_{p,q})=(t^{-1},p).
\end{equation*}
\end{thm}
We also reformulate these results as optimal embeddings into the scale 
of Lorentz spaces (cf. Definition \ref{Lpq}):
\begin{thm} \label{thm1.3}
(i) Let $1\le p<\infty$ and $0<q\le \min(p,2)$. Then 
\begin{equation*}
 B^0_{p,q}(\R^n)\hookrightarrow L_{p}(\R^n).
\end{equation*} 
(ii) Let $0<p<1$ and $0<q\le 1$. Then 
\begin{equation}\label{emb}
 B^{\sigma_p}_{p,q}(\R^n)\hookrightarrow L_{1,q}(\R^n).
\end{equation}
(iii) Let $0<p<1$ and $0<q\le\infty$. Then
\begin{equation*}
 F^{\sigma_p}_{p,q}(\R^n)\hookrightarrow L_{1,p}(\R^n)
\end{equation*}
and all these embeddings are optimal with respect to the second fine parameter of the scale of the Lorentz spaces.
\end{thm}
\begin{rem}
(i) Let us also observe, that \eqref{emb} improves \cite[Theorem 3.2.1]{SiTr} and 
\cite[Theorem 2.2.3]{SR}, where the embedding 
$B^{n(\frac 1p-1)}_{p,q}(\R^n)\hookrightarrow L_{1}(\R^n)$ is proved for all $0<p<1$ and $0<q\le 1.$

(ii) Let us also mention, that growth envelopes for function spaces with minimal smoothness were
recently studied in \cite{CGO}. These authors worked with spaces defined by differences and
therefore their results are of a different nature.
\end{rem}

\section{Preliminaries, notation and definitions}
We use standard notation: $\N$ denotes the collection of all natural numbers,
$\R^n$ is the Euclidean $n$-dimensional space, where $n\in\N$, and $\C$ stands for the complex plane.

\begin{dfn}\label{Lpq}
(i) Let $0< p\le\infty$. We denote by $L_p(\R^n)$ the Lebesgue spaces endowed with the quasi-norm
$$
 ||f|L_p(\R^n)||=\begin{cases}
  \displaystyle\biggl(\int_{\R^n}|f(x)|^pd x\biggr)^{1/p},\quad& 0< p<\infty,\\
  \displaystyle\esssup_{x\in\R^n} |f(x)|,\quad &p=\infty.
 \end{cases}
$$
(ii) Let $0<p,q\le\infty$. Then the Lorentz space $L_{p,q}(\R^n)$ consists of all $f\in L_1^{\rm loc}(\R^n)$
such that the quantity
$$
||f|L_{p,q}(\R^n)||=\begin{cases}
\displaystyle\left(\int_0^\infty [t^{\frac 1p}f^*(t)]^q\frac{dt}{t}\right)^{1/q},\quad 0<q<\infty,\\
\displaystyle \sup_{0<t<\infty} t^{\frac 1p}f^*(t),\quad q=\infty
\end{cases}
$$
is finite
\end{dfn}
\begin{rem}
These definitions are well-known, we refer to \cite[Chapter 4.4]{BS} for details and further references.
We shall need only very few properties of these spaces. Obviously, $L_{p,p}(\R^n)=L_p(\R^n)$. 
If $0< q_1 \le q_2 \le\infty$, then $L_{p,q_1}(\R^n)\hookrightarrow L_{p,q_2}(\R^n)$ - so the Lorentz spaces
are monotonically ordered in $q$. We shall make use of the following lemma:
\end{rem}
\begin{lem}\label{lem1q}
Let $0<q<1$. Then the $||\cdot|L_{1,q}(\R^n)||$ is the $q$-norm, it means
$$
 ||f_1+f_2|L_{1,q}(\R^n)||^q\le  ||f_1|L_{1,q}(\R^n)||^q+ ||f_2|L_{1,q}(\R^n)||^q
$$
holds for all $f_1,f_2\in L_{1,q}(\R^n).$
\end{lem}
\begin{proof}
First note, that the function $s\to s^q$ is increasing for all $0<q<\infty$ on $(0,\infty)$. 
This leads to the identity 
\begin{equation}\label{eqq}
 (|f|^q)^*(t)= (f^*)^q(t),
\end{equation}
which holds for all $t>0$, $0<q<\infty$ and all measurable functions $f$. The reader may also consult
\cite[Proposition 2.1.7]{BS}. Using \eqref{eqq} and $0<q<1$, we obtain
\begin{align*}
||f_1+f_2|L_{1,q}(\R^n)||^q&=
\int_0^\infty t^{q-1}((f_1+f_2)^*(t))^qdt
\le \int_0^\infty t^{q-1}((|f_1|+|f_2|)^*(t))^qdt\\
&= \int_0^\infty t^{q-1}((|f_1|+|f_2|)^q)^*(t)dt
\le \int_0^\infty t^{q-1}(|f_1|^q+|f_2|^q)^*(t)dt.
\end{align*}
We observe, that $t:\to t^{q-1}$ is a decreasing function on $(0,\infty)$ and that 
$$
 \int_0^\xi (|f_1|^q+|f_2|^q)^*(t) dt\le \int_0^\xi (|f_1|^q)^*(t)dt +\int_0^\xi (|f_2|^q)^*(t) dt
$$
holds for all $\xi\in(0,\infty)$.
Hence, by Hardy's lemma (cf. \cite[Proposition 2.3.6]{BS}),
\begin{align*}
||f_1+f_2|L_{1,q}(\R^n)||^q&\le 
\int_0^\infty t^{q-1}(|f_1|^q)^*(t)dt+\int_0^\infty t^{q-1}(|f_2|^q)^*(t)dt\\
&=||f_1|L_{1,q}(\R^n)||^q+ ||f_2|L_{1,q}(\R^n)||^q.
\end{align*}
\end{proof}

Let $S(\R^n)$ be the Schwartz space of all complex-valued rapidly decreasing, infinitely differentiable
functions on $\R^n$ and let $S'(\R^n)$ be its dual - the space of all tempered distributions.

For $f\in S'(\R^n)$ we denote by $\widehat f= F f$ its Fourier transform and by 
$f^\vee$ or $F^{-1}f$ its inverse Fourier transform.

We give a Fourier-analytic definition of Besov and Triebel-Lizorkin spaces, which relies
on the so-called {\it dyadic resolution of unity}. Let $\phii\in S(\R^n)$ with
\begin{equation}\label{eq0}
 \phii(x)=1\quad \text{if}\quad |x|\le 1\quad\text{and}\quad\phii(x)=0\quad\text{if}\quad|x|\ge \frac 32.
\end{equation}We put $\phii_0=\phii$ and $\phii_j(x)=\phii(2^{-j}x)-\phii(2^{-j+1}x)$ for $j\in\N$ and $x\in\R^n.
$
This leads to the identity
\begin{equation*}
 \sum_{j=0}^\infty\phii_j(x)=1,\qquad x\in\R^n.
\end{equation*}
\begin{dfn}\label{defsp}
(i) Let $s\in\R, 0< p,q\le\infty$. Then $B^{s}_{pq}(\R^n)$ is the collection
of all $f\in S'(\R^n)$ such that
\begin{equation}\label{eq1}
 ||f|B^s_{pq}(\R^n)||=\biggl(\sum_{j=0}^\infty 2^{jsq}||(\phii_j \widehat f)^\vee|L_p(\R^n)||^q\biggr)^{1/q}<\infty
\end{equation}
(with the usual modification for $q=\infty$).

(ii) Let $s\in\R, 0< p<\infty, 0< q\le\infty$. Then $F^{s}_{pq}(\R^n)$ is the collection
of all $f\in S'(\R^n)$ such that
\begin{equation}\label{eq2}
 ||f|F^s_{pq}(\R^n)||=\biggl|\biggl|\biggl(\sum_{j=0}^\infty 2^{jsq}|(\phii_j \widehat f)^\vee(\cdot)|^q\biggr)^{1/q}|L_p(\R^n)\biggr|\biggr|<\infty
\end{equation}(with the usual modification for $q=\infty$).
\end{dfn}
\begin{rem}
These spaces have a long history. In this context we recommend \cite {P}, \cite{T-FS1}, \cite{T-FS2} and \cite{T-FS3}
as standard references. We point out that the spaces $B^s_{pq}(\R^n)$ and $F^s_{pq}(\R^n)$
are independent of the choice of $\phii$ in the sense of equivalent (quasi-)norms.
Special cases of these two scales include Lebesgue spaces, Sobolev spaces,
H\"older-Zygmund spaces and many other important function spaces.
\end{rem}

We introduce the sequence spaces associated with the Besov and Triebel-Lizorkin spaces.
Let $m\in\Z^n$ and $j\in\N_0$. Then $Q_{j\, m}$ denotes the closed cube in $\R^n$
with sides parallel to the coordinate axes, centred at $2^{-j}m$,
and with side length $2^{-j}$. By $\chi_{j\, m}=\chi_{Q_{j\,m}}$ we denote the characteristic
function of $Q_{j\,m}$. If
$$
\lambda=\{\lambda_{j\,m}\in\C:j\in\N_0, m\in\Z^n\},
$$
$-\infty<s<\infty$ and $0<p,q\le \infty$, we set
\begin{equation}\label{eq:2.10}
||\lambda| b^s_{pq}||=\biggl(\sum_{j=0}^\infty 2^{j (s-\frac np)q}\Bigl(\sum_{m\in\Z^n}|\lambda_{j\, m}|^p
\Bigr)^{\frac qp}\biggr)^\frac 1q
\end{equation}
appropriately modified if $p=\infty$ and/or $q=\infty$. If $p<\infty$, we define also
\begin{equation}\label{defspqf'}
||\lambda|f^{s}_{pq}||=\biggl|\biggl|
\biggl(\sum_{j=0}^\infty \sum_{m\in\Z^n}|2^{j s}\lambda_{j\,m}\chi_{j\,m}(\cdot)|^q\biggr)^{1/q}
|L_p(\R^n)\biggr|\biggr|.
\end{equation}
The connection between the function spaces $B^s_{pq}(\R^n)$, $F^s_{pq}(\R^n)$ and
the sequence spaces $b^s_{pq}$, $f^s_{pq}$ may be given by various decomposition
techniques, we refer to \cite[Chapters 2 and 3]{T-FS3} for details and further references.

All the unimportant constants are denoted by the letter $c$, whose meaning may differ
from one occurrence to another. If $\{a_n\}_{n=1}^\infty$ and $\{b_n\}_{n=1}^\infty$
are two sequences of positive real numbers, we write $a_n\lesssim b_n$ if, and only if, there
is a positive real number $c>0$ such that $a_n\le c\, b_n, n\in\N.$ Furthermore,
$a_n\approx b_n$ means that $a_n\lesssim b_n$ and simultaneously $b_n\lesssim a_n$.

\section{Proofs of the main results}

\subsection{Proof of Theorem \ref{thm1.2} (i)}

In view of Theorem \ref{thm1.1}, it is enough to prove, that for $1\le p<\infty$ 
and $0<q\le\min(p,2)$ the index $u$ associated to $B^0_{p,q}(\R^n)$ is greater or equal to $p$.

We assume in contrary that \eqref{eq:index} is fulfilled for some $0<v<p$, $\epsilon>0$,
$c>0$ and all $f\in B^0_{p,q}(\R^n)$. 
Let $\psi$ be a non-vanishing $C^{\infty}$ function in $\R^n$ supported in $[0,1]^n$
with $\int_{\R^n}\psi(x)dx =0.$

Let $J\in\N$ be such that $2^{-Jn}<\epsilon$
and consider the function
\begin{equation}\label{atoms}
 f_j=\sum_{m=1}^{2^{(j-J)n}}\lambda_{j\, m}\psi (2^j(x-(m,0,\dots,0))),\quad j>J,
\end{equation}
where 
$$
 \lambda_{j\, m}=\frac{1}{m^{\frac 1p}\log^{\frac 1v} (m+1)},\quad m=1,\dots, 2^{(j-J)n}
$$
Then \eqref{atoms} represents an atomic decomposition of $f$ in the space 
$B^0_{p,q}(\R^n)$ according to \cite[Chapter 1.5]{T-FS3} and we obtain (recall that $v<p$)
\begin{align}\notag
 ||f_j|B^0_{p,q}(\R^n)||&\lesssim
 2^{-j\frac np}\left(\sum_{m=1}^{2^{(j-J)n}}\lambda_{j\, m}^p\right)^{1/p}
 \le 2^{-j\frac np}\left(\sum_{m=1}^\infty m^{-1}(\log(m+1))^{-\frac pv}\right)^{1/p}\\
 &\lesssim 2^{-j\frac np}\label{eq:small}.
\end{align}
On the other hand,
\begin{align*}
\left(\int_0^\epsilon \left[f_j^*(t)t^{\frac 1p}\right]^v\frac{dt}{t}\right)^{1/v}
&\ge \left(\int_0^{2^{-Jn}}f_j^*(t)^v t^{v/p-1} dt\right)^{1/v}
\gtrsim \left(\sum_{m=1}^{2^{(j-J)n}}\lambda_{j\, m}^v\int_{c\,2^{-jn}(m-1)}^{c\,2^{-jn}m} t^{v/p-1}dt\right)^{1/v}\\
&\gtrsim \left(\sum_{m=1}^{2^{(j-J)n}}\lambda_{j\, m}^v 2^{-jnv/p}m^{v/p-1}  \right)^{1/v}
=2^{-j\frac{n}{p}}\left(\sum_{m=1}^{2^{(j-J)n}}\frac{1}{m \log(m+1)}\right)^{1/v}.
\end{align*}
As the last series is divergent for $j\to\infty$, this is in a contradiction with \eqref{eq:small} and 
\eqref{eq:index} cannot hold for all $f_j,j>J.$

\begin{rem}
Observe, that Theorem \ref{thm1.3} (i) is a direct consequence of Theorem \ref{thm1.2} (i). The embeddings
$B^0_{1,q}(\R^n)\hookrightarrow B^0_{1,1}(\R^n)\hookrightarrow L_1(\R^n)$ if $p=1$
and
$B^0_{p,q}(\R^n)\hookrightarrow F^0_{p,2}(\R^n)=L_p(\R^n)$ if $1<p<\infty$ show, that
$B^0_{p,q}(\R^n)\hookrightarrow L_p(\R^n)$. And Theorem \ref{thm1.2} (i) implies that if 
$B^0_{p,q}(\R^n)\hookrightarrow L_{p,v}(\R^n)$ for some $0<v< \infty$, then $p\le v.$
This proves the optimality of Theorem \ref{thm1.3} (i) in the frame of the scala of Lorentz spaces.
\end{rem}

\subsection{Proof of Theorem \ref{thm1.2} (ii) and Theorem \ref{thm1.3} (ii)}

Let $0<p<1$, $0<q\le 1$ and $s=\sigma_p=n\left(\frac 1p-1\right)$. We prove first 
Theorem \ref{thm1.3} (ii), i.e. we show that
$$
 B^{\frac np-n}_{p,q}(\R^n)\hookrightarrow L_{1,q}(\R^n),
$$
or, equivalently,
$$
 \left(\int_0^\infty [tf^*(t)]^q\frac{dt}{t}\right)^{1/q}\le c\, ||f|B^{\frac np-n}_{p,q}(\R^n)||,
 \qquad f\in B^{\frac np-n}_{pq}(\R^n).
$$
Let 
$$
 f=\sum_{j=0}^\infty f_j=\sum_{j=0}^\infty \sum_{m\in \Z^n} \lambda_{j\, m} a_{j\, m}
$$
be the optimal atomic decomposition of an $f\in B^{\frac np-n}_{p,q}(\R^n)$, again in the sense of 
\cite{T-FS3}.
Then
\begin{equation}\label{eq:p1}
 ||f|B^{\frac np-n}_{p,q}(\R^n)||
 \approx \left(\sum_{j=0}^\infty 2^{-jqn}\left(\sum_{m\in\Z^n}|\lambda_{j\, m}|^p\right)^{q/p}\right)^{1/q}
\end{equation}
and by Lemma \ref{lem1q} 
\begin{equation}\label{eq:p2}
 ||f|L_{1,q}(\R^n)||= ||\sum_{j=0}^\infty f_j|L_{1,q}(\R^n)||
 \le \left(\sum_{j=0}^\infty ||f_j|L_{1,q}(\R^n)||^q\right)^{1/q}.
\end{equation}
We shall need only one property of the atoms $a_{j\, m}$, namely that their support is contained
in the cube $\tilde Q_{j\, m}$ - a cube centred at the point $2^{-j}m$ with sides parallel to the coordinate
axes and side length $\alpha 2^{-j}$, where $\alpha>1$ is fixed and independent of $f$. 
We denote by $\tilde\chi_{j\, m}(x)$ the characteristic functions of $\tilde Q_{j\, m}$
and by $\chi_{j\, l}$ the characteristic function of the interval $(l2^{-jn},(l+1)2^{-jn}).$
Hence
$$
 f_j(x)\le c\sum_{m\in\Z^n} |\lambda_{j\, m}|\tilde\chi_{j\, m}(x),\quad x\in\R^n
$$
and
\begin{align}
\notag
||f_j|L_{1,q}(\R^n)||
&\lesssim \left(\int_0^\infty \sum_{l=0}^\infty \left[(\lambda_j)^*_l \chi_{j\, l}(t)\right]^q t^{q-1}dt\right)^{1/q}
\le \left(\sum_{l=0}^\infty \left[(\lambda_j)^*_l\right]^q
\int_{2^{-jn}l}^{2^{-jn}(l+1)}t^{q-1}dt\right)^{1/q}\\
&\lesssim 2^{-jn} \left(\sum_{l=0}^\infty \left[(\lambda_j)^*_l\right]^q(l+1)^{q-1}\right)^{1/q}
\lesssim 2^{-jn}||\lambda_j|\ell_p||.
\label{eq:p3}
\end{align}
The last inequality follows by $(l+1)^{q-1}\le 1$ and $\ell_p\hookrightarrow \ell_q$ if $p\le q.$
If $p>q$, the same follows by H\"older's inequality with respect to indices $\alpha=\frac pq$
and $\alpha'=\frac{p}{p-q}$:
$$
\left(\sum_{l=0}^\infty \left[(\lambda_j)^*_{l}\right]^q (l+1)^{q-1}\right)^{1/q}
\le \left(\sum_{l=0}^\infty \left[(\lambda_j)^*_{l}\right]^{q\cdot\frac{p}{q}}\right)^{\frac 1q\cdot\frac{q}{p}}
\cdot \left(\sum_{l=0}^\infty (l+1)^{(q-1)\cdot \frac{p}{p-q}}\right)^{\frac{1}{q}\cdot\frac{p-q}{p}}
\le c\,||\lambda_j|\ell_p||.
$$
Here, we used that for $0<q<p<1$ the exponent $\frac{(q-1)p}{p-q}=-1+\frac{(p-1)q}{p-q}$ is strictly 
smaller than $-1$.

The proof now follows by \eqref{eq:p1}, \eqref{eq:p2} and \eqref{eq:p3}.
$$
 ||f|L_{1,q}(\R^n)||\le \left(\sum_{j=0}^\infty ||f_j|L_{1,q}(\R^n)||^q\right)^{1/q}
 \le c\, \left(\sum_{j=0}^\infty 2^{-jnq}||\lambda_j|\ell_p||^q\right)^{1/q}
 \le c\, ||f|B^{\sigma_p}_{p,q}(\R^n)||.
$$

\begin{rem}
We actually proved, that \eqref{eq:index} holds for $X=B^{\frac np-n}_{pq}(\R^n)$, $v=q$ and 
$\epsilon=\infty$. This, together with Theorem \ref{thm1.1} (iii) implies immediately Theorem \ref{thm1.2} (ii).
\end{rem}

\subsection{Proof of Theorem \ref{thm1.2} (iii) and Theorem \ref{thm1.3} (iii)}
Let $0<p<1$ and $0<q\le\infty$. By the Jawerth embedding (cf. \cite{J} or \cite{V})
and Theorem \ref{thm1.2} (ii) we get for any $0<p<\tilde p<1$
$$
 F^{\sigma_p}_{p,q}(\R^n)\hookrightarrow B^{\sigma_{\tilde p}}_{\tilde p,p}(\R^n)\hookrightarrow
 L_{1,p}(\R^n).
$$

\end{document}